\documentclass[letterpaper, 10 pt, conference]{ieeeconf}
\IEEEoverridecommandlockouts % This command is only needed if you want to use the \thanks command
\usepackage{tabularx}

\usepackage{algorithm}
\usepackage{algpseudocode}
\usepackage{blindtext}
\usepackage{graphicx}
\usepackage[utf8]{inputenc}
\usepackage{graphics} % for pdf, bitmapped graphics files
\usepackage{epsfig} % for postscript graphics files
\usepackage{mathptmx} % assumes new font selection scheme installed
\usepackage{amsmath} % assumes amsmath package installed
\usepackage{amssymb}  % assumes amsmath package installed
\usepackage{cite}
\usepackage{multicol}
\usepackage{mathtools, cuted}
\usepackage[cmintegrals]{newtxmath}
\usepackage{epstopdf}
\usepackage{graphics} % for pdf, bitmapped graphics files
\usepackage{epsfig} 

\usepackage{epsfig} % for postscript graphics files
\usepackage{mathptmx} % assumes new font selection scheme installed
\usepackage{amsmath} % assumes amsmath package installed
\usepackage{blkarray, bigstrut}
\usepackage{amssymb}  % assumes amsmath package installed
\usepackage{cite}
\usepackage{multicol}
\usepackage{mathtools, cuted}
\usepackage[cmintegrals]{newtxmath}
\usepackage{makeidx}

\usepackage{multicol}        % used for the two-column index
\usepackage[bottom]{footmisc}% places footnotes at page bottom

\usepackage{nomencl}
\usepackage[usenames, dvipsnames]{color}

\usepackage{algorithm}
\usepackage{algpseudocode}

\usepackage{amsthm}
\theoremstyle{definition}

\newtheorem{theorem}{Theorem}

\newtheorem{assumption}{Assumption}

\theoremstyle{remark}

\title{\LARGE \bf
A Continuous-Time Optimal Control Approach to Congestion Control
}

\author{ Harshvardhan Uppaluru, Hamid Emadi, and Hossein Rastgoftar
\thanks{H. Uppaluru, H. Emadi and H. Rastgoftar are with the Aerospace and Mechanical Engineering Department at University of Arizona. Emails: \{huppaluru, hamidemadi, hrastgoftar\}@email.arizona.edu}}

\begin{document}

\maketitle

\thispagestyle{empty}
\pagestyle{empty}

\begin{abstract}

Traffic congestion has become a nightmare to modern life in metropolitan cities. \textit{On average, a driver spending \textit{X} hours a year stuck in traffic} is one of most common sentences we often read regarding traffic congestion. Our aim in this article is to provide a method to control this seemingly ever-growing problem of traffic congestion. We model traffic dynamics using a continuous-time mass-flow conservation law, and apply optimal control techniques to control traffic congestion. First, we apply the mass-flow conservation law to specify traffic feasibility and present continuous-time dynamics for modeling traffic as a network problem by defining a network of interconnected roads (NOIR). The traffic congestion control is formulated as a boundary control problem and we use the concept of state-transition matrix to help with the optimization of boundary flow by solving a constrained optimal control problem using quadratic programming. Finally, we show that the proposed algorithm is successful by simulating on a NOIR.

\end{abstract}

\section{INTRODUCTION}

Urbanization and rapid increase in the usage of private vehicles has led to the problem of urban traffic congestion becoming prominent in almost every city and developing into a global issue. One of the major reasons behind traffic congestion is the inefficient use of the urban traffic networks. Traffic congestion continues to have a significant negative impact on the economy \cite{muneera2018economic}, and the environment \cite{ye2013road} \cite{annan2015traffic} due to increase in vehicle emissions, degrading air quality and posing significant health risks  \cite{chor2011impact} \cite{o2004impact}. Focused on improving mobility, saving energy, understanding and influencing travel behavior, traffic control is a significant and active research area in the field of Intelligent Transportation Systems (ITS). A number of methods dealing with prediction, control and optimization of traffic congestion and variety of approaches such as model-based \& model-free have been proposed by researchers to reduce and control traffic congestion.

The traditional light-based approach to deal with automated operation of traffic signals at junctions is called fixed-cycle control. To optimize traffic signal timings, a standard fixed-cycle control tool called traffic network study tool has been used \cite{robertson1969transyt} \cite{tiwari2008continuity}. To optimize for green time interval at junctions, fuzzy-based signal control method was employed \cite{balaji2011type} \cite{chiu1992adaptive}. A number of physics-based approaches have been proposed that make use of the Fundamental Diagram to determine traffic state \cite{zhang2012ordering} \cite{zhang2011transitions}. Link-based Kinematic Wave model (LKWM) was developed to model dynamic traffic coordination in continuous-time\cite{han2016continuous}. Spillback congestion was incorporated in \cite{gentile2007spillback} \cite{adamo1999modelling}. Hierarchical fuzzy-based systems and genetic algorithms were the basis for the novel approach proposed in \cite{zhang2014hierarchical} to build traffic congestion prediction systems. A model based on changes in driving behavior that does not rely on traffic flow monitoring infrastructure was proposed in \cite{ito2017predicting} thereby forecasting traffic congestion. 

%\begin{figure}[ht]
%    \centering
%    \includegraphics[width=\linewidth]{road_network.png}
%    \caption{Example NOIR}
%    \label{fig:examplenoir}
%\end{figure}

Inspired by mass-flow conservation, \cite{jafari2018decentralized} developed first order traffic dynamics. A popular choice of model-based approach for optimization of traffic coordination is Model Predictive Control (MPC). MPC is a model-based feedback control technique relying on real-time optimization. \cite{lin2012efficient} provided a structured network-wide traffic controller that was capable of coordinating an urban traffic network. \cite{jamshidnejad2017sustainable} developed an MPC system that used a gradient-based optimization approach to find a solution to the traffic control optimization problem. Other methods which have been applied to deal with model-based traffic management are Neural Networks (NN) \cite{kumar2015short, akhter2016neural, tang2017improved, moretti2015urban}, Markov Decision Process (MDP) \cite{ong2017markov, haijema2008mdp, liu2021boundary, liu2021conservation, rastgoftar2020integrative, rastgoftar2019integrative, rastgoftar2020resilient}, Formal Methods \cite{coogan2017formal} \cite{coogan2015traffic}, Mixed Non-Linear Programming (MNLP) \cite{christofa2013person}, and Optimal Control \cite{jafari2018decentralized} \cite{wang2018dynamic}. 

% \subsection{Contributions}
This paper proposes a continuous-time approach for modeling and control of traffic in a network of interconnected roads (NOIR). We first  apply mass-flow conservation and obtain a new model for dynamics of traffic coordination which is presented by a stochastic process and  governed by a first-order differential equation. Traffic congestion control is then defined as a boundary control problem with the control input representing the boundary inflow and state aggregating traffic density across the NOIR. The boundary inflow is optimized by  solving a constrained optimal control problem with the cost penalizing the traffic density across the network. We define the control constraints  such that  feasibility of the model is assured, while the proposed traffic modeling and control assures avoidance of traffic backflow. We also use the Fundamental Diagram  to impose the traffic feasibility conditions such that the traffic congestion can be minimized by solving a constrained continuous-time optimal control problem.
% and (ii) the traffic feasibility conditions are continuously updated and incorporated into planning and control of traffic coordination. 
% \begin{enumerate}
%     \item 
% \end{enumerate}
% Our contributions in this article are towards a novel continuous-time optimal control approach to coordinate and control traffic congestion in a network of interconnected roads (NOIR). Traffic coordination is modeled as a mass-flow conservation problem thus advancing the previous works \cite{liu2021boundary} \cite{liu2021conservation}  \cite{rastgoftar2020integrative} \cite{rastgoftar2019integrative} \cite{rastgoftar2020resilient} by considering the problem in continuous-time to determine optimal boundary flow.

This paper is organized as follows: Section II is the preliminary section that introduces the concept of a NOIR. The problem statement is explained along with the assumptions in Section III. The continuous-time dynamics behind the traffic network is presented in Section IV. Our continuous-time optimal traffic control approach is described in Section V. We finally present our simulation results using the described traffic model and optimal control approach on a NOIR in Section VI before putting forward our concluding remarks in Section VII.

\section{Preliminaries}

A NOIR describes a finite set of serially-connected road elements, where $i \in \mathcal{V}$ represents each unique road element. Let us define $\mathcal{V}_{in}$, $\mathcal{V}_{out}$, and $\mathcal{V}_I$ as the sets consisting of the index numbers of the inlet, outlet, and interior road elements{\color{blue},} respectively. Therefore, $\mathcal{V}_{in} = \{1,\cdots,N_{in}\}$, $\mathcal{V}_{out} = \{N_{in}+1,\cdots,N_{out}\}$ and $\mathcal{V}_I = \{N_{out}+1,\cdots,N\}$. $N_{in}$, $N_{out}$, and $N_I$ correspond to the total number of inlets, outlets and interior road elements. The complete set of road elements can be represented as $\mathcal{V} = \mathcal{V}_{in} \cup \mathcal{V}_{out} \cup \mathcal{V}_I=\left\{1,\cdots,N\right\}$. Therefore, $N$ corresponds to the total number of road elements in the NOIR. The interactions between road elements are established by graph $\mathcal{G}(\mathcal{V}, \mathcal{E})$, where $\mathcal{V}$ corresponds to the vertices and $\mathcal{E} \subset \mathcal{V} \times \mathcal{V}$ corresponds to the edges. A connection directed from road element $i \in \mathcal{V}$ to road element $j \in \mathcal{V}$ is represented as the edge $(i, j) \in \mathcal{E}$. An \textit{in-neighbor set} $\mathcal{I}_i \triangleq \{ j | (j,i) \in \mathcal{E} \} \subset \mathcal{V}_{in} \cup \mathcal{V}_I$ specifies upstream adjacent road elements for every road element $i \in \mathcal{V}$. Similarly, for every road element $i \in \mathcal{V}$, \textit{out-neighbor set} $\mathcal{O}_i \triangleq \{ j | (i,j) \in \mathcal{E} \} \subset \mathcal{V}_{out} \cup \mathcal{V}_I$ specified downstream adjacent road elements. 

\begin{figure}[ht]
    \centering
    \includegraphics[width=\linewidth]{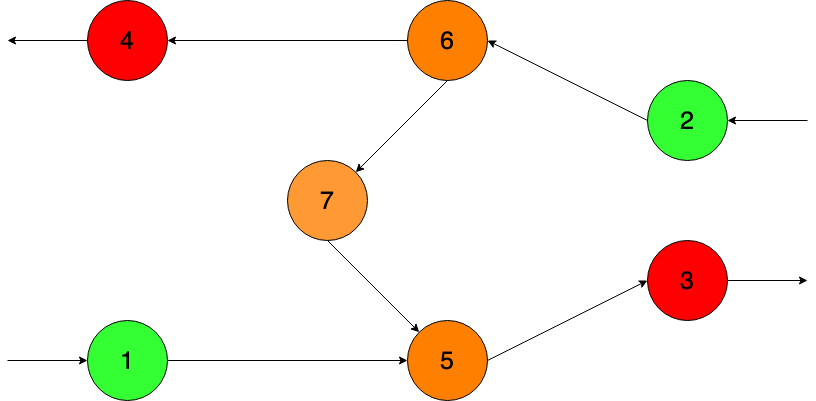}
    \caption{An example of a NOIR. Green nodes represent inlet roads. Red nodes represent outlet roads. Orange nodes represent interior road elements.}
    \label{fig:simple_noir}
\end{figure}

Consider the example of a simple NOIR presented in Fig \ref{fig:simple_noir}. Since there are $2$ inlets, we have the set $\mathcal{V}_{in} = \{1, 2\}$ and $N_{in} = 2$. Similarly, there are $2$ outlets, we have $\mathcal{V}_{out} = \{3, 4\}$ and $N_{out} = 2$. Finally, there are $3$ interior road elements, we have $\mathcal{V}_I = \{5, 6, 7\}$ and $N_I = 3$. The complete set $\mathcal{V} = \mathcal{V}_{in} \cup \mathcal{V}_{out} \cup \mathcal{V}_I = \{1,\cdots,7\}$ and $N = 7$. The graph $\mathcal{G}(\mathcal{V}, \mathcal{E})$ is defined by $\mathcal{E} = \{ (1,5),(5,3),(2,6),(6,4),(6,7),(7,5) \}$.

\begin{assumption}\label{assump1}
For the purposes of this paper, we assume that, for every inlet element $i \in \mathcal{V}_{in}$, $\mathcal{I}_i = \emptyset$ and $|\mathcal{O}_i| = 1$. Correspondingly, for every outlet element $j \in \mathcal{V}_{out}$, we assume that $|\mathcal{I}_j| = 1$ and $|\mathcal{O}_j| = \emptyset$. This assumption was originally proposed in \cite{rastgoftar2020integrative}.
\end{assumption}

\section{Problem Statement}
For every road element $i\in \mathcal{V}$,  $u_i(t)$, $y_i(t)$, $\rho_i(t)$, and $z_i(t)$ denote the internal inflow, internal outflow, traffic density, and external outflow, respectively. We apply mass conservation law and model traffic coordination in road $i\in \mathcal{V}_I$ by

\begin{equation} \label{rho_dot}
    \dot{\rho}_i=y_i-z_i
\end{equation}
where
\begin{subequations}
\begin{equation} \label{y_i}
    y_i(t)=\sum_{j\in \mathcal{O}_i}q_{j,i}y_i,\qquad t\in\left[t_0,t_f\right]
\end{equation}
\begin{equation} \label{z_i}
    z_i(t)=p_i\rho_i(t),\qquad t\in \left[t_0,t_f\right]
\end{equation}
\end{subequations}
where $p_i\in \left(0,1\right]$ is the outflow probability of road $i\in \mathcal{V}$; $q_{j,i}\in \left[0,1\right]$ is the tendency probability specifying the fraction of outflow of road $i$ directed from $i$ to $j\in \mathcal{O}_i$, where $t_0$ and $t_f$ are fixed \textit{initial} and \textit{final} times, respectively.

\begin{assumption}\label{assump2}
This paper assumes that $p_i$ and $q_{i,j}$ remain constant over the time interval $\left[t_0,t_f\right]$.
\end{assumption} 

\begin{assumption}\label{assump3}
This paper assumes that $\dot{\rho}_i(t)=0$ at any time $t\in \left[t_0,t_f\right]$, if $i\in \mathcal{V}_{in}\bigcup \mathcal{V}_{out}$, i.e. $y_i(t)=z_i(t)$ and $\rho_i(t)$ remain constant at any time $t\in \left[t_0,t_f\right]$, if  $i\in \mathcal{V}_{in}\bigcup \mathcal{V}_{out}$.
\end{assumption} 

\begin{assumption}\label{assump4}
This paper assumes that $u_i(t)=0$ at any time $t\in \left[t_0,t_f\right]$, if $i\in \mathcal{V}_I\bigcup \mathcal{V}_{out}$.
\end{assumption} 

The objective of this paper is to determine boundary control $u_i$ at every boundary inlet road $i\in \mathcal{V}_{in}$ such that the traffic cost function 
\begin{equation}\label{cost}
    C={\frac{1}{2}}\int_{t_0}^{t_f}\left(\sum_{i\in\mathcal{V}_{I}}r_i\rho_i^2(t)+\sum_{j\in\mathcal{V}_{in}}w_ju_j^2(t)\right)dt
\end{equation}
is minimized, constraint \eqref{stateconstraint} and the following input, inequality and equality constraints are satisfied:
\begin{subequations}
\begin{equation}\label{C1}
    \bigwedge_{i\in\mathcal{V}_{in}}\left(u_i(t)>0\right),\qquad \forall t\in \left[t_0,t_f\right],
\end{equation}
\begin{equation}\label{C2}
    \bigwedge_{i\in\mathcal{V}\in  \mathcal{V}_{in}\bigcup\mathcal{V}_{out}}\left(u_i(t)=0\right),\qquad \forall t\in \left[t_0,t_f\right],
\end{equation}
\begin{equation}\label{C3}
    \sum_{i\in\mathcal{V}_{in}}u_i(t)=u_0,\qquad \forall t\in \left[t_0,t_f\right],
\end{equation}
\end{subequations}
where $r_i>0$ and $w_j\geq0$ are constant scaling factors, for $i\in \mathcal{V}_I$ and $j\in \mathcal{V}_{in}$, and the net boundary inflow $u_0$ is constant. The inequality \eqref{C1} specifies a feasibility condition to assure that back-flow is avoided at every inlet boundary road $i\in \mathcal{V}_{in}$. Constraint \eqref{C2} implies that the external flow is $0$, if $i\in \mathcal{V}_{I}\bigcup\mathcal{V}_{out}$ (see Assumption \ref{assump3}). Constraint \eqref{C3} assures that the net inflow to the NOIR is constant at any time $t\in \left[t_0, t_f\right]$. This condition is imposed to ensure that $u_0$ cars are permitted to enter the NOIR when the demand for using the NOIR is high.

\begin{assumption}

\begin{figure}[ht]
    \centering
    \includegraphics[width=\linewidth]{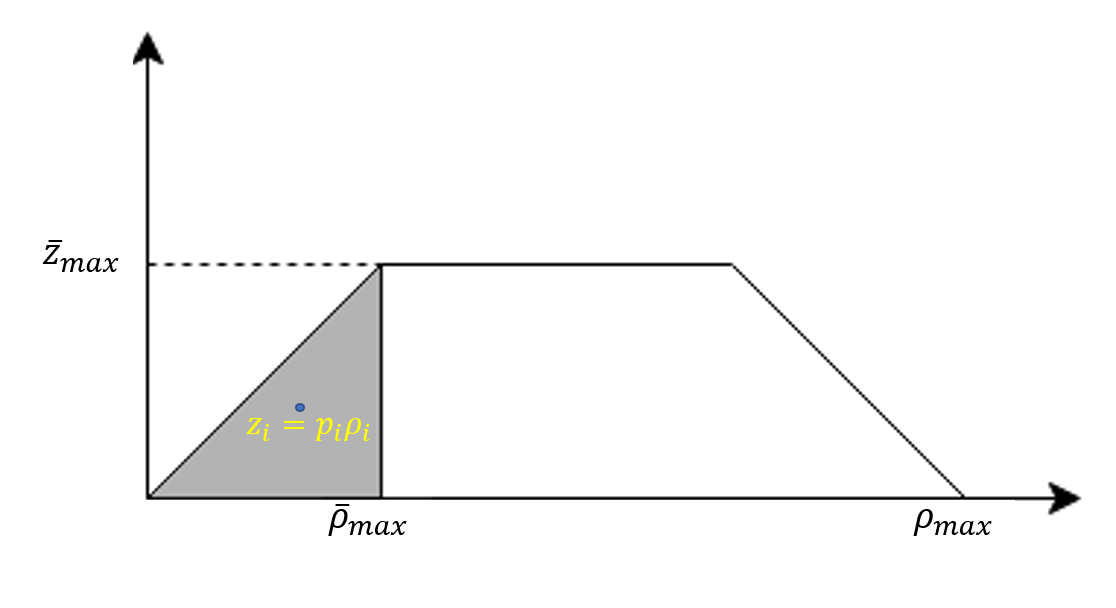}
    \caption{Schematic of the Fundamental Diagram}
    \label{fig:assum5}
\end{figure}

In this paper, we use the Fundamental Diagram to impose the following feasibility condition
\begin{equation}\label{stateconstraint}
    \bigwedge_{i\in \mathcal{V}}\left(x_i(t)\leq \bar{\rho}_{max}\right),\qquad \forall t\in \left[t_0,t_f\right]
\end{equation}
where $\bar{\rho}_{max}$ is assigned by the Fundamental Diagram (see Fig. \ref{fig:assum5}). Therefore, the traffic outflow probability $p_i$ ($i\in \mathcal{V}_I$) must satisfies the following inequality condition:
\begin{equation}
    p_i\leq {\bar{z}_{max}\over \bar{\rho}_{max}},\qquad \forall i\in \mathcal{V}_I,
\end{equation}
where $\bar{z}_{max}$ is the maximum out flow for every road $i\in \mathcal{V}_I$. 
\end{assumption}

\section{Traffic Network Dynamics}

According to Assumption \ref{assump3}, traffic density remains constant at inlet and outlet roads. Therefore, the traffic dynamics are only defined for the interior road elements. To model traffic coordination, we define the state vector $\mathbf{x} = [\rho_{N_{out+1}},\cdots,\rho_{N}]^{T} \in \mathbb{R}^{(N - N_{out}) \times 1}$, boundary input vector $\mathbf{u}=\begin{bmatrix}u_1&\cdots&u_{N_{in}}\end{bmatrix}^T\in \mathbb{R}^{N_{in}\times 1}$, the inflow vector $\mathbf{y} \in \mathbb{R}^{(N - N_{out}) \times 1}$, and the outflow vector $\mathbf{z} \in \mathbb{R}^{(N - N_{out}) \times 1}$. We also define the positive definite and diagonal outflow probability matrix $\mathbf{P}=\mathbf{diag}\left(p_{N_{out}+1},\cdots,p_N\right) \in \mathbb{R}^{(N - N_{out}) \times (N - N_{out})}$ and the non-negative tendency probability matrix $\mathbf{Q}=\left[Q_{ij}\right] \in \mathbb{R}^{(N - N_{out}) \times (N - N_{out})}$ that is defined as follows:
% that are defined as follows:

\begin{equation}
Q_{ij}=\begin{cases}
q_{i+N_{out},j+N_{out}}&\left(i+N_{out}\right)\in \mathcal{O}_{j+N_{out}}\\
0&\mathrm{otherwise}
\end{cases}.
\end{equation}
% \end{subequations}
By considering \eqref{rho_dot}, \eqref{y_i}, and \eqref{z_i}, we can relate $\mathbf{y}$ and $\mathbf{z}$ to $\mathbf{x}$ by
\begin{subequations}
\begin{equation}
    \mathbf{y}=\mathbf{Q}\mathbf{P}\mathbf{x},
\end{equation}
\begin{equation}
    \mathbf{z}=\mathbf{P}\mathbf{x},
\end{equation}
\end{subequations}
and model the network traffic dynamics by the following dynamics:
\begin{equation}\label{DiscreteeeDynamicssssss}
    \dot{\mathbf{x}}\left(t\right)=\mathbf{A}\mathbf{x}\left(t\right)+\mathbf{B}\mathbf{u}\left(t\right),\qquad \forall t\in \left[t_0,t_f\right],
\end{equation}
where
\begin{equation} \label{eq:A}
    \mathbf{A}=\left(\mathbf{Q}-\mathbf{I}\right)\mathbf{P},
\end{equation}
$\mathbf{I} \in \mathbb{R}^{(N-N_{out}) \times (N-N_{out})}$, and $\mathbf{B}=\left[b_{ij}\right] \in \mathbb{R}^{(N-N_{out}) \times N_{in}}$ is defined as follows:

\begin{equation} \label{eq:B}
     b_{ij}=
    \begin{cases}
    1&j\in \mathcal{I}_{i+N_{out}}\\
    0&\mathrm{otherwise}
    \end{cases}
    .
\end{equation}

\begin{theorem}\label{thm1}
Assume graph $\mathcal{G}$  defining the NOIR interconnections has the following properties:
\begin{enumerate}
    \item{There exists at least a path from every boundary inlet node $j\in \mathcal{V}_{in}$ towards $i\in \mathcal{V}_I$.}
    \item{There exists at least a path from $i\in \mathcal{V}_I$ towards every boundary outlet node $h\in \mathcal{V}_{out}$.}
\end{enumerate}
Then, Matrix $\mathbf{A}$ is Hurwitz with  eigenvalues that are all placed inside a unit  disk centered at $-1+0j$. 
% Also, dynamics \eqref{DiscreteeeDynamicssssss} is BIBO stable and eigenvalues of matrix $\bar{\mathbf{Q}}_{D,k}$  are all placed inside a disk of radius $r_{Q,k}$ centered at the origin.
\end{theorem}
\begin{proof}
If assumptions of Theorem \ref{thm1} are satisfied, eigenvalues of matrix $\mathbf{Q}$ are strictly inside the unit disk centered at the origin. Therefore,  eigenvalues of matrix $\mathbf{Q}-\mathbf{I}$ are strictly inside the unit disk centered at the $-1+0\mathbf{j}$ which in turn implies that matrix $\mathbf{Q}-\mathbf{I}$ is Hurwitz. Because $\mathbf{P}$ is positive definite and diagonal with diagonal elements that are all less than or equal to $1$, eigenvalues of matrix $\mathbf{A}=\left(\mathbf{Q}-\mathbf{I}\right)\mathbf{P}$ are strictly inside the unit disk centered at the $-1+0\mathbf{j}$ and $\mathbf{A}$ is Hurwitz.
\end{proof}
Theorem \ref{thm1} implies that the traffic network dynamics  \eqref{DiscreteeeDynamicssssss} is bounded-input-bounded-output (BIBO) stable.

\section{Traffic Control}
The objective of the traffic control is to determine the optimal boundary input $\mathbf{u}^*$ such that the traffic coordination cost, defined by \eqref{cost}, is minimized and the traffic feasibility conditions \eqref{C1}, \eqref{C2}, and \eqref{C3} are all satisfied. This problem can be formalized as follows:
\begin{equation}
    \min{\frac{1}{2}}\int_{t_0}^{t_f}\left(\mathbf{x}^T\mathbf{R}\mathbf{x}+\mathbf{u}^T\mathbf{W}\mathbf{u}\right)dt
\end{equation}
subject to
\begin{subequations}
\begin{equation}\label{CC1}
    \mathbf{u}\geq 0,
\end{equation}
% \begin{equation}
%   \mathbf{1}_{1\times N_{in}} \mathbf{u}=u_0,
% \end{equation}
\begin{equation}\label{CC2}
   \mathbf{1}_{1\times N_{in}} \mathbf{u}=u_0,
\end{equation}
\end{subequations}
where $t_0$ and $t_f$ are fixed, $\mathbf{x}_0=\mathbf{x}\left(t_0\right)$ is given, $\mathbf{x}_f=\mathbf{x}\left(t_f\right)$ is free, $\mathbf{W}=\mathbf{diag}\left(w_1,\cdots,w_{N_{in}}\right)\in \mathbb{R}^{N_{in}\times N_{in}}$ is diagonal and positive definite, and $\mathbf{R}=\mathbf{diag}\left(r_{N_{out}+1},\cdots,r_{N}\right)\in \mathbb{R}^{\left(N-N_{out}\right)\times\left(N-N_{out}\right)}$ is diagonal and positive semi-definite. To solve the above constrained optimal control problem, we first define Hamiltonian as 
\begin{equation}
    {H}\left(\mathbf{x},\mathbf{u},\mathbf{\lambda}\right)={\frac{1}{2}}\left(\mathbf{x}^T\mathbf{R}\mathbf{x}+\mathbf{u}^T\mathbf{W}\mathbf{u}\right)+\lambda^T\left(\mathbf{A}\mathbf{x}+\mathbf{B}\mathbf{u}\right),
\end{equation}
where $\lambda\in \mathbb{R}^{\left(N-N_{out}\right)\times 1}$ is the co-state vector. By imposing necessary conditions from Table 3.2-1 in \cite{lewis2012optimal}, $\dot{\mathbf{x}}=\bigtriangledown_\lambda H$ and $\dot{\mathbf{\lambda}}=-\bigtriangledown_\mathbf{x} H$, $\mathbf{x}_{sys}=\begin{bmatrix}\left(\mathbf{x}^*\right)^T&\left(\lambda^*\right)^T\end{bmatrix}^T$ is updated by the following dynamics:
\begin{equation}
    \dot{\mathbf{x}}_{sys}(t)=\mathbf{A}_{sys}{\mathbf{x}}_{sys}(t)+\mathbf{B}_{sys}{\mathbf{u}}^*(t),\qquad t\in \left[t_0,t_f\right]
\end{equation}
subject to $\mathbf{x}\left(t_0\right)=\mathbf{x}_0$ and $\lambda\left(t_f\right)=\mathbf{0}$, where 
\begin{subequations}
\begin{equation}
    \mathbf{A}_{sys}=\begin{bmatrix}\mathbf{A}&\mathbf{0}\\-\mathbf{R}&-\mathbf{A}^T\end{bmatrix}\in \mathbb{R}^{2\left(N-N_{out}\right)\times 2\left(N-N_{out}\right)},
\end{equation}
\begin{equation}
    \mathbf{B}_{sys}=\begin{bmatrix}\mathbf{B}\\\mathbf{0}\end{bmatrix}\in \mathbb{R}^{2\left(N-N_{out}\right)\times N_{in}},
\end{equation}
\end{subequations}
and ${\mathbf{u}}^*$ is assigned by solving the following optimization problem
\begin{equation}\label{eq:u_star}
    \mathbf{u}^*=\min\limits_\mathbf{u} {H}\left(\mathbf{x}^*,\mathbf{u},\mathbf{\lambda}^*\right)
\end{equation}
subject to constraints \eqref{CC1} and \eqref{CC2}.

Defining the state transition matrix by the following equation
\begin{equation}\label{eq:phi}
    \Phi(t, t_0) = e^{A_{sys}(t-t_0)}
\end{equation}
we can write the solution to the system in terms of state transition matrix as follows
\begin{equation}\label{eq:x_sys}
    \mathbf{x}_{sys}(t) = \Phi(t, t_0)\mathbf{x}_{sys}(t_0) + \Psi(t, t_0),
\end{equation}
where
\begin{equation}\label{eq:psi}
    \Psi(t, t_0) = \int_{t_0}^{t}\Phi(\epsilon, t_0)\mathbf{B}_{sys}\mathbf{u}(\epsilon)d\epsilon.
\end{equation}
Eq. \eqref{eq:x_sys} can also be written in matrix form as

\begin{equation}\label{eq:x_lambda}
    \begin{bmatrix}
    \mathbf{x}(t) \\
    \lambda(t)
    \end{bmatrix} = \begin{bmatrix}
    \Phi_{11}(t, t_0) & \Phi_{12}(t, t_0) \\
    \Phi_{21}(t, t_0) & \Phi_{22}(t, t_0)
    \end{bmatrix}\begin{bmatrix}
    \mathbf{x}_0 \\
    \lambda_0
    \end{bmatrix} + \begin{bmatrix}
    \Psi_{1}(t, t_0) \\
    \Psi_{2}(t, t_0)
    \end{bmatrix}
\end{equation}

Using the above formulations, we can find $\lambda_0$ as
\begin{equation}\label{eq:l_0}
    \lambda_{0} = - (\Phi_{22})^{-1}(\Phi_{21}\mathbf{x}_0 + \Psi_{2})
\end{equation}

Substituting $\lambda_0$ in \eqref{eq:x_lambda} iteratively for a predetermined number of iterations, we will obtain the optimal boundary inflow $\mathbf{u}^*$ (See Algorithm \ref{alg:ctoccc}).

\begin{algorithm}
\caption{Continuous-Time Optimal Control to Congestion Control Algorithm}\label{alg:ctoccc}

\begin{algorithmic}[1]
\Require A Network of Interconnected Roads (NOIR)
\State Randomly initialize $\mathbf{P}$ and $\mathbf{Q}$
\State Obtain $\mathbf{A}$ \& $\mathbf{B}$ using \eqref{eq:A} \& \eqref{eq:B} respectively
\State $R \gets I$, $M \gets N_{in}$, $N \gets 2(N-N_{out})$
\State Choose $m$ \& $n$
\State Initialize $\lambda$, $\mathbf{x}$, $\mathbf{x}_0$
\For {$i \gets 1,\cdots,m$}
    \For{$j \gets 1,\cdots,n$}
        \State $H \gets I \in \mathbb{R}^{M \times M}$
        \State $f \gets \mathbf{B}\lambda(j)$
        \State Compute $\mathbf{u}(j)$ \eqref{eq:u_star} using MATLAB's quadprog
    \EndFor
    \State Compute $\Psi$, $\Phi$ using \eqref{eq:phi} \& \eqref{eq:psi} respectively
    \State Compute $\lambda_0$ using \eqref{eq:l_0}
    \For{$j \gets 1,\cdots,n$}
        \State Compute $\Phi$, $\Psi$ using \eqref{eq:phi} \& \eqref{eq:psi} respectively
        \State $\mathbf{x}(j, i) \gets \Phi \mathbf{x}_0 + \Psi$
        \State $\lambda(j, i) \gets \mathbf{x}(N+1:2*N, j, i)$
    \EndFor
\EndFor
\end{algorithmic}
\end{algorithm}

\section{Simulation Results}

\begin{figure}[ht]
    \centering
    \includegraphics[width=\linewidth]{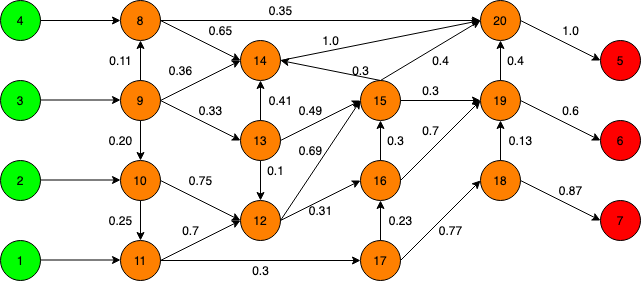}
    \caption{A network of interconnected roads (NOIR)}
    \label{fig:example_noir}
\end{figure}

We consider the NOIR originally presented in \cite{rastgoftar2020integrative}, that consists of $20$ unidirectional roads shown in Fig \ref{fig:example_noir}. Following the approach mentioned in Section II, we have the complete node set $\mathcal{V} = \{1,\cdots,20\}$ which can be represented as $\mathcal{V} = \mathcal{V}_{in} \cup \mathcal{V}_{out} \cup \mathcal{V}_I$ where $\mathcal{V}_{in} = \{1,\cdots,4\}$, $\mathcal{V}_{out} = \{5,\cdots,7\}$, $\mathcal{V}_I = \{8,\cdots,20\}$. Therefore $N_{in} = | \mathcal{V}_{in} | = 4$, $N_{out} = | \mathcal{V}_{out} | = 3$ and $N_I = | \mathcal{V}_I | = 13$. Hence, $N = 20$.

The outflow probabilities for each of the interior road elements $8,\cdots,20$ are $\bar{p}_{8} = 0.67$, $\bar{p}_{9} = 0.76$, $\bar{p}_{10} = 0.71$, $\bar{p}_{11} = 0.59$, $\bar{p}_{12} = 0.67$, $\bar{p}_{13} = 0.94$, $\bar{p}_{14} = 0.94$, $\bar{p}_{15} = 0.83$, $\bar{p}_{16} = 0.69$, $\bar{p}_{17} = 0.58$, $\bar{p}_{18} = 0.97$, $\bar{p}_{19} = 0.96$ and $\bar{p}_{20} = 0.91$. The matrices $\mathbf{Q}$ and $\mathbf{P}$ are obtained from the probabilities shown in the Fig. \ref{fig:example_noir}. Matrix $\mathbf{Q}$ and $\mathbf{P}$ are of the shape ($N_I$, $N_I$). Using (\ref{eq:A}) and (\ref{eq:B}), we determine the traffic tendency matrix $\mathbf{A}$ and $\mathbf{B}$ respectively. We choose $m=15$, $n=2000$ for optimization purposes. We follow the approach mentioned in Algorithm \ref{alg:ctoccc}.

\begin{figure}[ht]
    \centering
    \includegraphics[width=\linewidth]{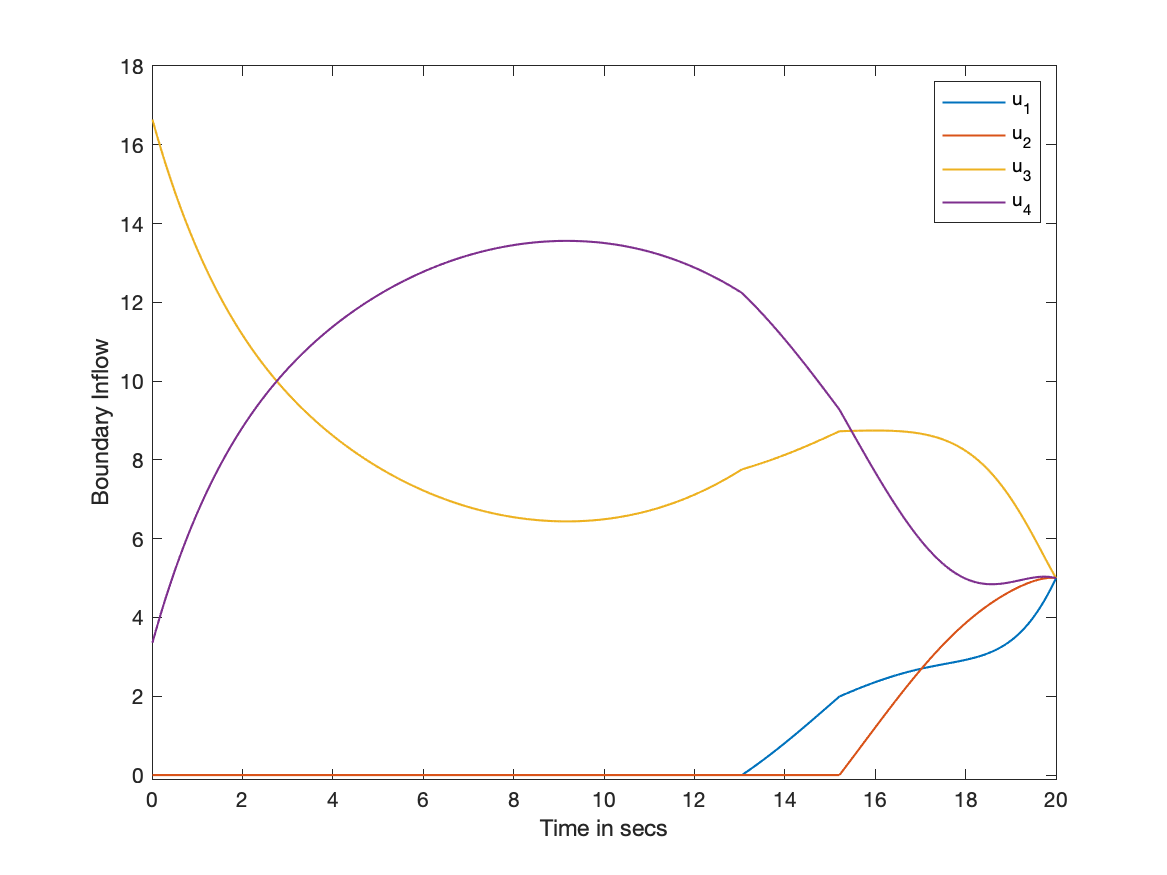}
    \caption{Optimal boundary inflow rate $u_1,\cdots,u_4$.}
    \label{fig:example_inflow}
\end{figure}

\begin{figure}[ht]
    \centering
    \includegraphics[width=\linewidth]{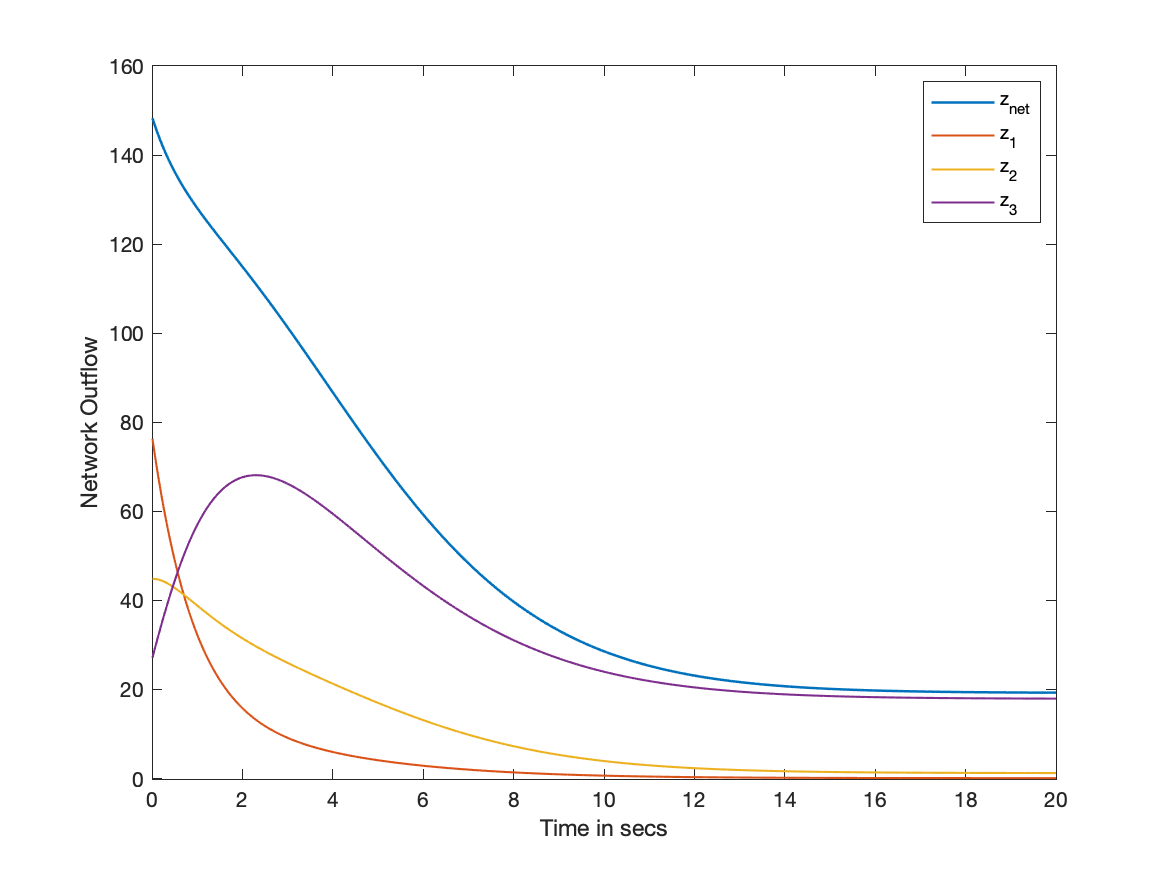}
    \caption{Optimal boundary outflow rate $z_5,\cdots,z_7$}
    \label{fig:example_outflow}
\end{figure}

\begin{figure}[ht]
    \centering
    \includegraphics[width=\linewidth]{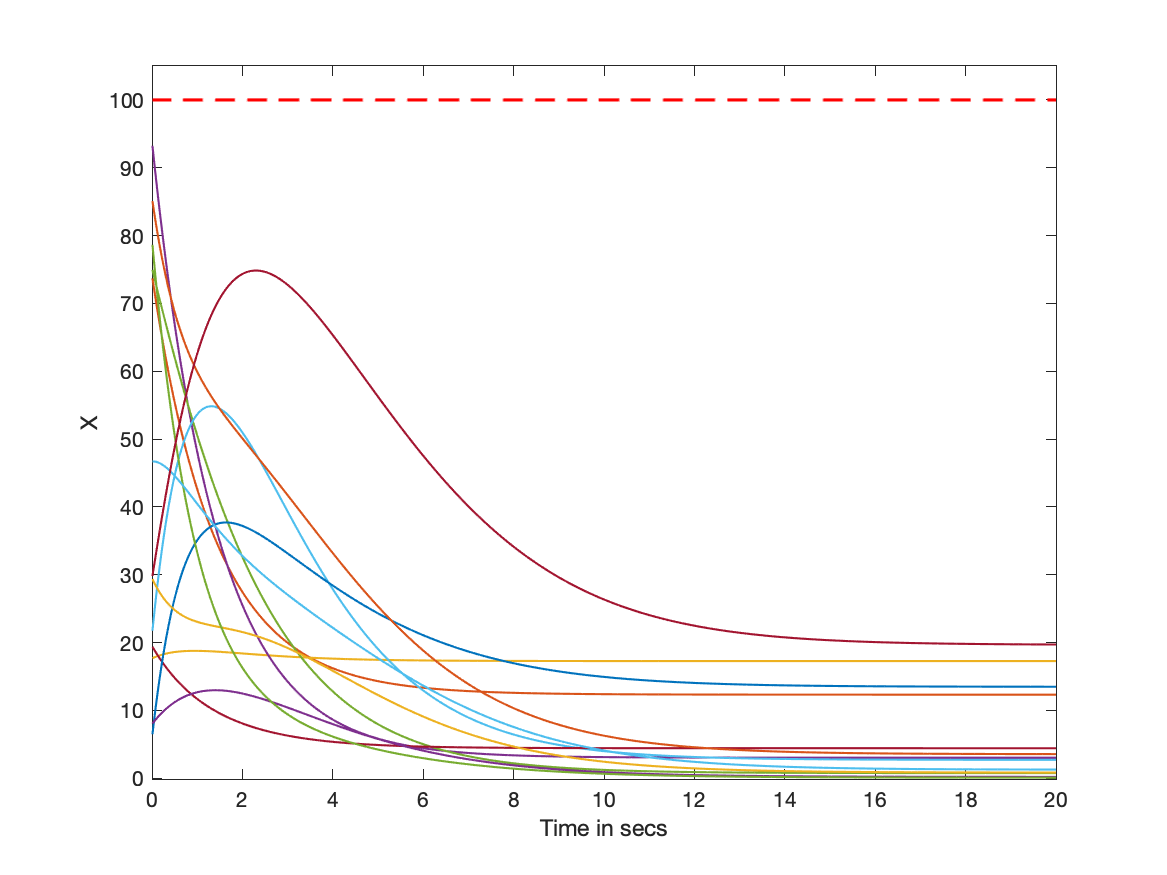}
    \caption{Traffic density at interior road elements}
    \label{fig:example_state}
\end{figure}

 The number of vehicles entering the NOIR is restricted at any time i.e., $u_0 = 20$. We plot the simulation results for the four inlets boundary road elements $1,\cdots,4 \in \mathcal{V}_{in}$ versus time in Fig \ref{fig:example_inflow}. We also plot the simulation results for the three outlet boundary road elements $5,\cdots,7 \in \mathcal{V}_{out}$ versus time in Fig \ref{fig:example_outflow}. We can observe from these two figures that the steady state in terms of traffic density is achieved after about 15 seconds. We can also see that at $t_f=20s$
 
 \begin{equation}
     u_1 \cong u_2 \cong u_3 \cong u_4 \cong 5
 \end{equation}
 
 \begin{equation}
     z_{net} = z_{5} + z_{6} + z_{7} \cong u_{0} \cong 20
 \end{equation}
 
 Fig \ref{fig:example_state} shows that the traffic densities at each of the $13$ interior road elements achieves steady-state values after about $15$ seconds.

\subsection{Relationship between $\lambda$ and $\mathbf{R}$}

Depending on how we initialized $\mathbf{R}$, we have observed that initial $\lambda$ values are also affected (though $\lambda(t_f) = 0$). For our purposes, we have used $\mathbf{R} =\zeta \mathbf{I} \in \mathbb{R}^{M \times M}$ with the multiplication factor $\zeta>0$. Fig \ref{fig:relation} shows that increasing $\zeta$ results in boosting  the maximum value of $\lambda_0$.

\begin{figure}[ht]
    \centering
    \includegraphics[width=\linewidth]{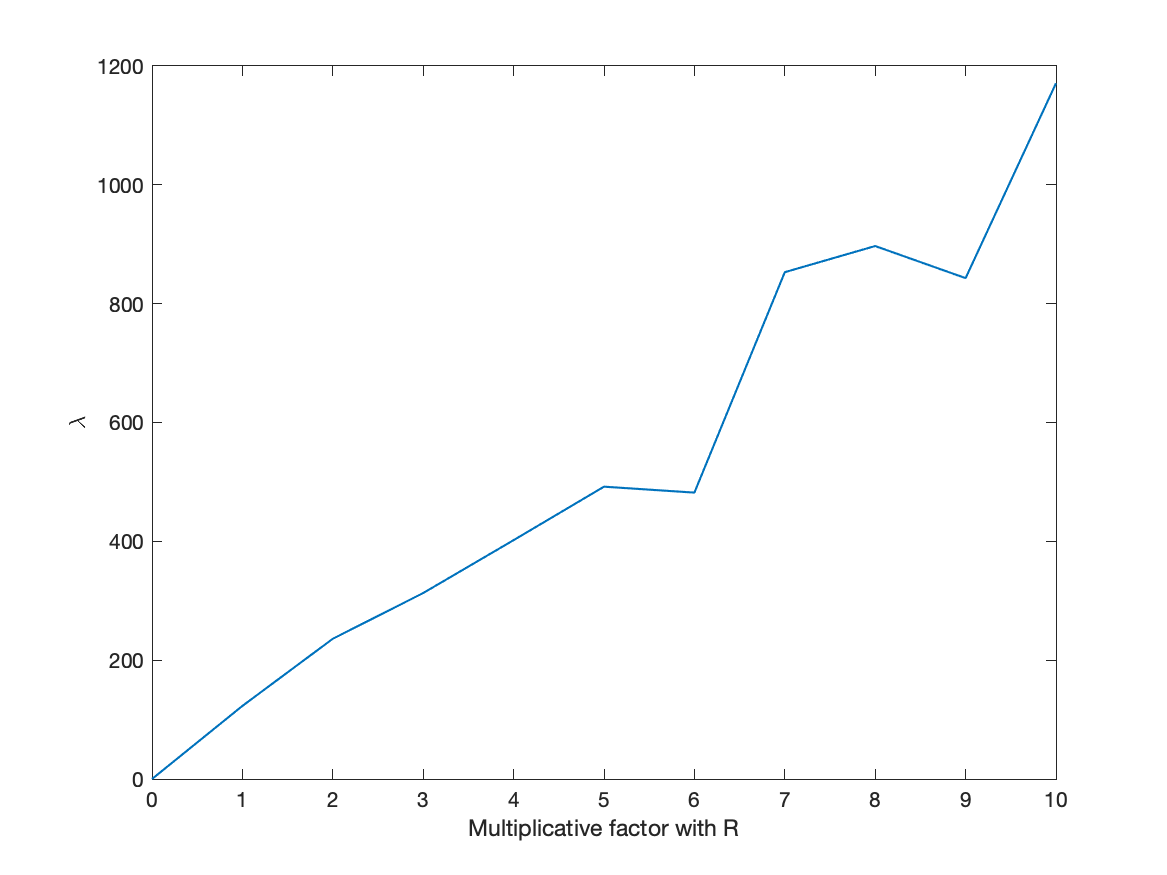}
    \caption{Relationship between $\lambda$ and $\mathbf{R}$}
    \label{fig:relation}
\end{figure}

\section{Conclusion}

This articles introduces a continuous-time optimal control approach to model and control traffic congestion that has been presented as an algorithm. In comparison with previous work, a continuous-time formulation of the modeling and dynamics of traffic congestion was presented. Simulation studies on a NOIR exhibits that the proposed model and control has been able to achieve the results through boundary control of traffic flow. Future work in this area include modeling and controlling traffic congestion efficiently as a Markov Decision Process.
   We also plan to  implementation on real-world traffic networks to learn and improve the traffic flow.

\section{Acknowledgement}
This work has been supported by the National Science
Foundation under Award Nos. 2133690 and 1914581. We would also like to thank Mr. Xun Liu at Villanova University for his help. 

\bibliographystyle{IEEEtran}
\bibliography{myref}

\end{document}